\documentclass[11pt]{amsart}

\usepackage{amsfonts, amstext, amsmath, amsthm, amscd, amssymb}
\usepackage{epsfig, graphics, psfrag}
\usepackage{color}
\usepackage[margin=1.2in]{geometry}
\usepackage{adjustbox}

\let\oldmarginpar\marginpar
\renewcommand\marginpar[1]{\oldmarginpar[\raggedleft\footnotesize #1]%
{\raggedright\footnotesize #1}}

\renewcommand{\setminus}{{\smallsetminus}}

\newcommand{\RR}{{\mathbb{R}}}
\newcommand{\Z}{{\mathbb{Z}}}

\newcommand{\vol}{{\rm vol}}

\newcommand{\lmin}{{\ell_{\rm min}}}

\theoremstyle{plain}
\newtheorem{theorem}{Theorem}[section]
\newtheorem{corollary}[theorem]{Corollary}
\newtheorem{lemma}[theorem]{Lemma}
\newtheorem{proposition}[theorem]{Proposition}

\newtheorem{conjecture}[theorem]{Conjecture}

\newtheorem*{namedtheorem}{\theoremname}
\newcommand{\theoremname}{testing}
\newenvironment{named}[1]{\renewcommand{\theoremname}{#1}\begin{namedtheorem}}{\end{namedtheorem}}

\theoremstyle{definition}
\newtheorem{definition}[theorem]{Definition}
\newtheorem{remark}[theorem]{Remark}

\title{Gromov norm and Turaev-Viro invariants of 3-manifolds}

\author{Renaud Detcherry}
\author{Efstratia Kalfagianni}

\address[]{Department of Mathematics, Michigan State University, East
Lansing, MI, 48824, USA}

\email[]{kalfagia@math.msu.edu}

\email[]{renaud.detcherry@gmail.com}

\begin{document}
%\vspace*{3 cm }
%\begin{center}
%\textbf{Abstract}
%\newline
%\newline 
%\end{center}

\thanks{Research supported by NSF Grants DMS-1404754 and DMS-1708249} \date{\today}
\maketitle

%%%%%%%%%%%%%%%%%%%%%%%%%%%%%%%%%%%%%%%%%%%%%%%%%%%%%%%%%%%%%%%%%%%%%%%%%%%%%%%%%%%%%%%%%%%%%%%%%

\begin{abstract} We establish a relation between the 
``large r" asymptotics of the Turaev-Viro invariants $TV_r$ and  the Gromov norm of 3-manifolds.
We show that for any orientable, compact  3-manifold $M$,  with (possibly empty) toroidal boundary,  $\log |TV_r (M)|$ is bounded above by a function linear in $r$ and whose slope is
a positive universal constant times the Gromov norm of $M$.
The proof  combines TQFT techniques, geometric decomposition theory of 3-manifolds and analytical estimates of $6j$-symbols.

We obtain topological criteria that can be used to check whether the growth is actually exponential; that is one has  $\log| TV_r (M)|\geqslant B \ r$, for some $B>0$.
We use these criteria  to construct infinite families of hyperbolic 3-manifolds whose $SO(3)$
Turaev-Viro invariants grow exponentially. 
These constructions  are essential for the results of  \cite{DK:AMU} where the authors  make progress on a conjecture of Andersen, Masbaum and Ueno about  the geometric properties of surface mapping class groups detected by the quantum representations.

We also study the behavior of the Turaev-Viro invariants under cutting and gluing of 3-manifolds along tori.
In particular, we show that, like the Gromov norm, the values of the invariants  do not increase under Dehn filling and we  give applications of this result on the question of the extent to which 
relations between the invariants  $TV_r$ and hyperbolic volume are preserved under Dehn filling.

Finally we give constructions of 3-manifolds, both with  zero and non-zero Gromov norm,  for which the Turaev-Viro invariants determine the Gromov norm.

 \end{abstract}

%%%%%%%%%%%%%%%%%%%%%%%%%%%%%%%%%%%%%%%

\section{Introduction}

Since the discovery of the  quantum 3-manifold invariants in the late 80's,  it has been a major challenge to understand their relations to the topology and geometry of 3-manifolds. Open conjectures predict tight connections between quantum invariants and the geometries coming from Thurston's 
geometrization picture \cite{book, Chen-Yang}. However, despite compelling physics and experimental evidence,  progress to these conjectures has been scarce. For instance, the volume conjecture
for the colored Jones polynomial  has only been verified for a handful of hyperbolic knots to date. The reader is referred to \cite{book} for survey articles on the subject and for
related conjectures.  On the other hand, coarse relations between the stable coefficients of  colored Jones polynomials and volume
have been established for an abundance of  hyperbolic knots \cite{DL, fkp:guts, fkp:filling}.

In this paper we are concerned with the question of
how the ``large level" asymptotics of the Turaev-Viro  3-manifold invariants relate to,  and interact with,  the geometric decomposition of 3-manifolds.
The Turaev-Viro invariants $TV_r(M)$ of a compact oriented $3$-manifold $M$  are 
combinatorially defined invariants that can be computed from triangulations of $M$ \cite{TuraevViro}.
They are real valued invariants,  indexed by a positive integer $r$, called the level,   and for each $r$ they depend on a $2r$-th root of unity.
We combine TQFT techniques, geometric decomposition theory of 3-manifolds and analytical estimates of $6j$-symbols  to show that the $r$-growth of $TV_r(M)$  is
bounded above by a function exponential in $r$ that involves the Gromov norm of $M$.

We also obtain topological criteria for the growth to be exponential; that is  to have $TV_r(M)\geqslant \exp{Br}$  with $B$ a positive constant.
We use these criteria  to construct infinite families of hyperbolic 3-manifolds whose $SO(3)$-Turaev-Viro invariants grow exponentially. 
These results are used by the authors \cite{DK:AMU} to make progress on a conjecture of Andersen, Masbaum and Ueno (AMU Conjecture) about the geometric properties of surface mapping class groups detected by the quantum representations.

\subsection{Upper bounds} For a compact oriented 3-manifold $M,$ 
 let  $TV_r( M,q)$ denote the $r$-th Turaev-Viro invariant of $M$ at  $q$, where $q$ is a $2r$-th root of unity such that $q^2$ is a primitive $r$-th root of unity. Throughout the paper we will work with
 $q=e^{\frac{2\pi i}{r}}$ and $r$ an odd integer and we will often write $TV_r( M):=TV_r(M, e^{\frac{2\pi i}{r}})$.  This is the theory that  corresponds to the  $SO(3)$ gauge group.  We define

\begin{equation}
LTV(M)=\underset{r\to \infty}{\limsup} \frac {2\pi}  {r} \log |TV_r(M))|
\label{suplim},
\end{equation}
and
\begin{equation}
lTV(M)=\underset{r\to \infty}{\liminf} \frac {2\pi}  {r} \log |TV_r(M))|
\label{suplim},
\end{equation}
where $r$ runs over all odd integers. Also we will use $ ||M||$ to denote the Gromov norm (or simplicial volume) of $M$. See Section \ref{sec:simplicialvol}  for definitions.
 The main result of this article is the following.
\begin{theorem}\label{LTVbound}
There exists a universal constant $C>0$ such that for any compact orientable $3$-manifold $M$ with empty or toroidal boundary we have
$$ LTV(M)\leqslant C||M||.$$
\end{theorem}

If the interior of $M$ admits a complete hyperbolic structure then, by Mostow rigidity, the hyperbolic metric is essentially unique
and the volume of the metric is a topological invariant denoted by $\vol (M)$, that is essentially the Gromov norm.
In this case, Theorem \ref{LTVbound} provides a relation between hyperbolic geometry and the Turaev-Viro invariants.
If $M$ is the complement of a hyperbolic link in $S^3$ then we know that $lTV(M)\geqslant 0$ and in many instances the inequality is strict (Corollary \ref{expgrowth} ).

The  problem of estimating the volume of hyperbolic 3-manifolds in terms of topological quantities and quantum invariants, 
has been studied considerably in the literature. See for example 
\cite{ast, fkp:guts, FKP} and references in the last item.  
 Despite  progress, to the best of our knowledge, Theorem  \ref{LTVbound} gives the first such linear  lower  bound that works for all hyperbolic  3-manifolds.
 
In the generality that Theorem \ref{LTVbound}  is stated, the constant $C$ is about $8.3581 \times 10^9$. However, within classes of 3-manifolds, one has much more effective
estimates. For instance, Theorem \ref{twistnumberbound} of this paper
shows that  for most (in a certain sense) hyperbolic links  $L\subset S^3$ we have
$$ LTV(S^3\setminus L )\leqslant 10. 5 \  \vol(S^3\setminus L ).$$
Furthermore, given any constant $E$ arbitrarily close to 1, one has infinite families of  hyperbolic closed and cusped 3-manifolds $M$, with $LTV(M)\leqslant \  E \ \vol(M)$. See Section \ref{twist}. 
for precise statements and more details.
  
 We also give   families of 3-manifolds with $ LTV(M)=||M||$.
One such  family  of examples is the class of links  with zero Gromov norm in $S^3$ or in $S^1\times S^2$, but we also   present families with non-zero norm (Section \ref{Exact}).
 \begin{corollary}Suppose that $M$ is $S^3$ or a connected sum of copies of  $S^1\times S^2$.
 Then, for any link   $K\subset M$ with $||M\setminus K||=0$, we have

$$lTV(M)=LTV(M)=\lim_{r\to \infty} \frac {2\pi} {r} {\log |TV_r(M\setminus K)|} =v_3 ||M\setminus K||=0,$$
where $r$ runs over all odd integers. 
\end{corollary}

  \subsection{Outline of proof of Theorem \ref{LTVbound}}
 A key step in the proof is to show that $LTV(M)$ is finite for any compact oriented $3$-manifold $M$.
 This is done 
by studying the large $r$ asymptotic behavior of the quantum $6j$-symbols, and using the state sum formulae for the invariants $TV_r$. 
More specifically, we prove the following.
\begin{theorem}\label{nbtetrahedrabound} Suppose that $M$ is a compact, oriented manifold with a triangulation consisting of  $t$ tetrahedra. Then, we have
$$LTV(M)\leqslant  2.08\  v_8\  t,$$

where $v_8\simeq  3.6638$ is the volume of  a regular ideal octahedron.\end{theorem}

A second  key argument we need is Theorem \ref{toruscutting} of the paper which
 describes the behavior of the Turaev-Viro invariants  under the operation of gluing or cutting 3-manifolds along tori. The proof of the theorem uses a version of a result  of Roberts  and Benedetti-Petronio  that
relates  $TV_r(M,e^{\frac{2\pi i}{r}})$ to the  $SO(3)$-Witten-Reshetikhin-Turaev  invariants,  and  it relies heavily on the properties of the corresponding 
TQFT as constructed by
Blanchet, Habegger, Masbaum and Vogel\,\cite{BHMV2}.

 By the Geometrization Theorem, a compact oriented  3-manifold $M$, with possibly
empty toroidal boundary can be cut along a canonical  collection of tori into pieces that are either Seifert fibered manifolds or hyperbolic (JSJ-decomposition).
We study the invariant $LTV$ for Seifert fibered manifolds using TQFT properties and Theorem \ref{toruscutting}.
Then we  prove Theorem \ref{LTVbound} by exploiting, by means of  Theorem \ref{toruscutting}, compatibility properties of $||M||$  and $LTV$ with the JSJ decomposition, and
 by studying separately  the case of hyperbolic manifolds using a theorem of Thurston.

\subsection{Lower bounds and the AMU Conjecture}  A very interesting problem, that we  will not address in this paper, is to prove the opposite inequality  of that given in Theorem \ref{LTVbound}.
We will discuss the weaker problem of exponential  $r$-growth of the invariants $TV_r(M)$.
Define
$$lTV(M)=\underset{r\to \infty}{\liminf} \frac {2\pi}  {r} \log |TV_r(M)|,$$
where $r$ runs over all odd integers.

In Section \ref{sec:cuttingtori} we show that, much  like the Gromov norm, the value of the invariant $lTV(M)$
does not increase under  the operation of  Dehn filling.  We also  discuss applications to the question of the extent to which relations between Turaev-Viro invariants and hyperbolic volume
are preserved under Dehn filling.  This in turn, leads to a topological criterion for checking whether 
the invariants $TV_r(M)$, of a  3-manifold $M$,  grow exponentially with respect to $r$. That is  checking 
whether  $lTV(M)>0$.  As a concrete application of this criterion,  combined with a result of \cite{DKY}, we mention the following.
More general results along these lines were obtained in \cite{BDKY}.

  \begin{corollary}\label{expgrowth}  Let $M\subset S^3$ denote the complement of  the figure-8  knot or the Borromean rings.
  For any link  $L\subset M$ we have 
  $$lTV(M\setminus L)\geqslant 2 v_3,$$
  
 where $v_3\simeq  1.0149$
  is volume of a regular ideal tetrahedron.
  \end{corollary}

Since, by \cite{DKY},  the invariants $TV_r(S^3\setminus L)$  of a link complement are expressed in terms of the
colored Jones polynomial of $L$, Corollary \ref{expgrowth} also provides new instances of colored Jones polynomial values with exponential growth. 
As far as we know this is the first instance where exponential growth of a quantum type invariant follows from a topological argument rather than brute force computations. 
The result
is consistent with the $TV_r$ volume  conjecture of Chen and Yang \cite{Chen-Yang}, which claims that for any hyperbolic 3-manifold of finite volume
we should have $lTV(M)=LTV(M)=\vol(M)$. See Section \ref{Exact} for more details.

Establishing exponential growth of the invariants $TV_r$ is also important for another intriguing  and wide-open conjecture  in quantum topology. This is  the AMU Conjecture due to Andersen, Masbaum  and Ueno \cite{AMU}.
In particular, Corollary \ref{expgrowth} has an essential application to this conjecture that  we  will explain  next.

For a compact  orientable 3-manifold of genus $g$ and $n$ boundary components, say
$\Sigma_{g, n}$, let $\mathrm{Mod}(\Sigma_{g, n})$ denote its  mapping class group.
The AMU conjecture asserts that  the $SU(2)$ and $SO(3)$ quantum
representations $\mathrm{Mod}(\Sigma_{g, n})$ should send pseudo-Anosov mapping classes to elements of infinite order (for large enough level). 
Despite good  progress on the AMU Conjecture for low genus surfaces \cite{ AMU, EgsJorgr, San17, San12}, the first examples that satisfy the conjecture in surfaces of genus at least two were recently given by
 March\'e and Santharoubane \cite{MarSan}.
 
In \cite{DK:AMU} the authors show that if we have $ lTV(M)>0$ for all hyperbolic 3-manifolds that fiber over the circle,
then the AMU Conjecture is true.
Corollary \ref{expgrowth} is then one of the key ingredients  used in \cite{DK:AMU}  to prove the following.

\begin{theorem}\label{AMU}  {\rm {({\cite{DK:AMU}})}}\label{general} 
Suppose that either $n=2$ and $g\geqslant 3$ or $g\geqslant n \geqslant 3.$
Then  there are infinitely many pseudo-Anosov mapping classes, up to conjugation and taking powers,  in
$\mathrm{Mod}(\Sigma_{g, n})$ that satisfy the AMU conjecture.
\end{theorem}

As far as we know Theorem \ref{AMU} is the first result that provides infinitely many mapping classes,  up to conjugation and taking powers,  that satisfy the AMU conjecture for fixed surfaces of genus at least $2.$

The paper is organized as follows: In Section \ref{sec:3-manifolds}, we recall some  results about the simplicial volume and the geometric decomposition of 3-manifolds. In Section \ref{sec:quantumdef}, we define the Turaev-Viro invariants and explain their TQFT properties. In Section \ref{sec:LTVfinite} we provide a bound for quantum $6j$-symbols that are used to define to $TV_r$ invariants, in terms of values of the  Lobachevsky function. In Section \ref{sec:cuttingtori} we study the behavior of  $LTV$ under the operations of cutting or gluing along tori. In Section \ref{sec:TVSeifert}, we study the special case of Seifert manifolds. In Section \ref{sec:LTVbound} we finish the proof of Theorem \ref{LTVbound} and  we derive Corollaries \ref{reduces} and \ref{expgrowth} and some generalizations.
Finally, in Section \ref{Exact}, we provide some new examples where the growth rate of Turaev-Viro invariants exactly computes the simplicial volume.

\vskip 0.08in

\noindent{\bf Acknowledgement.} We thank Gregor Masbaum, Cliff Taubes and Tian Yang for their interest in this work and for helpful comments and  discussions. 

%%%%%%%%%%%%%%%%%%%%%

\section{Decompositions of $3$-manifolds}
\label{sec:3-manifolds}
\subsection {Gromov norm preliminaries}
\label{sec:simplicialvol}
In this section, we recall the definition of the Gromov norm (a.k.a. simplicial volume) of $3$-manifolds and some of its classical properties.
 Gromov defined simplicial volume of $n$-manifolds in \cite{gromov:volume}, here we restrict ourselves to orientable 3-manifolds only. For more details the reader is referred to \cite[Section 6.5]{thurston:notes}.

\begin{definition}\label{simplicialvol} { \rm \cite{gromov:volume, thurston:notes}
Let $M$ be a compact orientable $3$-manifold with empty or toroidal boundary. Consider the fundamental class $[M, \partial M]$ in the singular homology $H_3(M, \partial M, \RR)$.
For $z=\sum c_i \sigma_i \in Z_3(M,\partial M, \RR),$ a $3$-relative singular cycle, representing $[M, \partial M]$,
we define its norm to be the real number $||z||=\sum |c_i|$.
\begin{enumerate} 
\item If  $\partial M=\emptyset$,  then the simplicial volume of $M$ is 
$$||M||=\mathrm{inf}\lbrace ||z|| \ / \ [z] = [M] \rbrace.$$
\item If $\partial M\neq \emptyset$, the representative  $[z]=[M, \partial M]$
determines a representative $\partial z$ of $[\partial M]\in H_2(\partial M, \RR)$. Then, as shown in \cite[Section 6.5]{thurston:notes} the following limit exists, 
$$||M||=\underset{\varepsilon \rightarrow 0}{\lim}\mathrm{inf}\lbrace ||z|| \ / \ [z] = [M,\partial M] \ \textrm{and} \ ||\partial z||\leqslant \varepsilon \rbrace,$$
and is defined to be the simplicial volume of $M$.
\end{enumerate}}
\end{definition}
\vskip 0.1in

 For hyperbolic manifolds, the simplicial volume is proportional to the hyperbolic volume and it is nicely behaved with respect to some topological operations.
\begin{theorem}\label{simplicialvolproperties} {\rm (\cite{gromov:volume, thurston:notes})} \ The following are true:
\begin{enumerate}
\item $||M||$ is additive under disjoint union and connected sums of manifolds.
\item If $M$ has a self map of degree $d> 1$ then $||M||=0.$ In particular $||\Sigma \times S^1||=0,$ for any compact oriented surface $\Sigma.$
\item If $T$  is an embedded torus in  $M$ and $M'$ is obtained from $M$ by cutting along $T$ then $$||M|| \leqslant ||M'||.$$
Moreover, the inequality is an equality if $T$ is incompressible in $M.$
\item If $M$ is obtained from $M'$ by Dehn-filling of a torus boundary component in $M',$ then
$$||M||\leqslant ||M'||.$$
\item If $M$ has a complete hyperbolic structure with finite volume then 
$$\vol (M)=v_3 ||M||.$$
\end{enumerate}
\end{theorem}
\subsection{Geometric decomposition}
\label{sec:JSJ}
We recall that any compact oriented $3$-manifold is a connected sum of irreducible manifolds and copies of $S^2\times S^1.$ 
Furthermore, by Jaco-Shalen-Johannson (JSJ) theory \cite{jaco-shalen, johannson}, any irreducible $3$-manifold $M$ can be cut along a canonical collection of incompressible tori $\mathcal{T}=\lbrace T_1 ,\ldots, T_n\rbrace,$ so that the components of $M\setminus \left( T_1 \cup \ldots \cup T_n \right)$ are irreducible atoroidal $3$-manifolds.

Thurston's Geometrization Conjecture \cite{thurston:survey}, proved by Perelman, allows one to identify the pieces of the JSJ decomposition.
A consequence of Perelman's work is the following theorem which is the solution to Thurston's Geometrization Conjecture.

\begin{theorem}\label{GT}
{\rm {(Geometrization Theorem, \cite{morgan-tian2})}} \ Any irreducible, compact, orientable $3$-manifold $M$ contains a unique (up to isotopy) collection of disjointly embedded incompressible tori ${\mathcal T}=\{T_1,\ldots, T_n\}$ such that all the connected components of 
$M\setminus {\mathcal T}$ are  either Seifert fibered manifolds  or hyperbolic.
\end{theorem}

\subsection{Efficient bounds on triangulations of $3$-manifolds}
\label{sec:triangulations}
We conclude this section by recalling a result  about triangulations of $3$-manifolds.
 As the Turaev-Viro invariants of a manifold $M$ are defined using state sums whose terms are products of quantum $6j$-symbols over a triangulation of $M,$ we wish to use triangulations with few tetrahedra to bound the Turaev-Viro invariants.
 For hyperbolic 3-manifolds, one way to achieve this is to consider  triangulations  not of the manifold $M$ itself, but rather of $M$ minus some geodesics.
  We will use the following theorem, due to W. Thurston, originally used in the proof of the so called Jorgensen-Thurston Theorem  \cite[Theorem 5.11.2]{thurston:notes}.
  
  \begin{theorem}\label{Thurstontriangulation}{\rm (Thurston)} \ There exists a universal constant $C_2,$ such that for any complete hyperbolic $3$-manifold $M$ of finite volume,  there exists a link $L$ in $M$ and a partially ideal triangulation of $M\setminus L$ with less than $C_2 ||M||$ tetrahedra.
\end{theorem}
The proof of Theorem \ref{Thurstontriangulation} comes from the thick-thin decomposition of hyperbolic manifolds.
Moreover, the constant $C_2$ in this theorem can be explicitly estimated: 

 It follows from the analysis in the proof of  \cite[Theorem 5.11.2]{thurston:notes} that given ${\displaystyle{\varepsilon\leqslant \frac{c}{2}}}$, where  $c$ is  the Margulis constant, 
one can choose 
$$C_2=\frac{\binom{k}{3}v_3}{4 V(\varepsilon/4)}\ \ {\rm  where} \ \  k=\lfloor \frac{V(5\varepsilon/4)}{V(\varepsilon/4)} \rfloor-1,$$ 
 where $V(r)$ denotes the hyperbolic volume of a ball of radius $r.$
The volume $V(r)$ be can be computed by the formula $V(r)=\pi \left( \sinh(2r)-2r\right)$ (see, for example,  \cite[Section 3.1]{Horvath}).
Moreover, the Margulis constant has been shown to be at least at least $0.104$ \cite{meyerhoff}.
Using ${\displaystyle {\varepsilon=\frac{0.103}{2}}}$ we get
 that in Theorem \ref{Thurstontriangulation}, we can use $C_2 =1.101\times 10^9.$

%%%%%%%%%%%

\section{Turaev-Viro invariants and  Witten-Reshetikhin-Turaev TQFT}
\label{sec:quantumdef}
In this section we summarize the definitions and the main properties of the quantum invariants we will use in this paper. First
we recall the definition of the Turaev-Viro invariants as state sums on triangulations of 3-manifolds. Then in subsection \ref{sec:TQFTdef} we summarize the properties
of the $SO(3)$-Witten-Reshetikhin-Turaev TQFT \cite{BHMV2, Roberts,  Turaevbook} that we will need in this paper.

\subsection{State sums for the Turaev-Viro invariants}
\label{sec:TVdef}

Let $r\geqslant 3$  be an odd integer and let $q=e^{\frac{2i\pi}{r}}.$
Define the quantum integer $\lbrace n \rbrace $ by
 $$\lbrace n \rbrace=q^n-q^{-n}=2\sin(\frac{2n\pi}{r})=2\sin(\frac{2\pi}{r})[n] , \ \ {\rm {where}} \ \   [n]=\frac{q^n-q^{-n}}{q-q^{-1}}=\frac{2\sin(\frac{2 n\pi}{r})}{2\sin(\frac{2\pi}{r})},$$

and define the quantum
 factorial by $\lbrace n \rbrace!=\underset{i=1}{\overset{n}{\prod}} \lbrace i \rbrace.$

 Consider the set  $I_r=\lbrace 0 ,2,4 \ldots, r-3 \rbrace$ of all non-negative even integers less than $r-2.$
A triple 
$(a_i, a_j, a_k)$  of elements  in $I_r,$ is called  {\emph admissible}  if $a_i+a_j+a_k\leqslant 2(r-2)$ and we have triangle inequalities  $a_i\leqslant a_j+a_k,$ $a_j\leqslant a_i+a_k,$ and $a_k\leqslant a_i+a_j.$
 For an admissible triple $(a_i,a_j,a_k),$ we define $\Delta(a_i,a_j,a_k)$  by 
$$\Delta(a_i,a_j,a_k)=\zeta_r ^{\frac{1}{2}}\left( \frac{\lbrace \frac{a_i+a_j-a_k}{2} \rbrace !\lbrace \frac{a_j+a_k-a_i}{2} \rbrace !\lbrace \frac{a_i+a_k-a_j}{2}\rbrace !}{\lbrace \frac{a_i+a_j+a_k}{2}+1 \rbrace !}\right)^{\frac{1}{2}}$$
where $\zeta_r={\displaystyle {2\sin(\frac{2\pi}{r})}}$.
A $6$-tuple $(a_1, a_2, a_3, a_4, a_5, a_6)\in I_r^6$ is called admissible if each of the triples
\begin{equation}
F_1=(a_1,a_2,a_3), \ \  F_2=(a_2,a_4,a_6), \ \ F_3=(a_1, a_5,a_6) \ \ {\rm and} \ \ F_4=(a_3,a_4,a_5), 
\label{notation}
\end{equation}
is admissible.
Given an admissible $6$-tuple   $(a_1, a_2, a_3, a_4, a_5, a_6)$,  we define  the quantum $6j$-symbol at the root $q$ by the formula

\begin{multline}\begin{vmatrix}
a_1 & a_2 & a_3 \\ a_4 & a_5 & a_6
\end{vmatrix}=
(\zeta_r)^{-1} (\sqrt{-1})^{\lambda} {\prod_{i=1}^4 \Delta(F_i)}   \sum_{z=\max \{T_1, T_2, T_3, T_4\}}^{\min\{ Q_1,Q_2,Q_3\}}\frac{(-1)^z\lbrace  z+1 \rbrace !}{{{\prod_{j=1}^4\lbrace z-T_j\rbrace !\prod_{k=1}^3\lbrace Q_k-z\rbrace !}}}
\label{6j}
\end{multline}
 where $\displaystyle{\lambda=\sum_{i=1}^6a_i,}$ and
 
 $$T_1=\frac{a_1+a_2+a_3}{2}, \ \ T_2=\frac{a_1+a_5+a_6}{2}, \ \  T_3=\frac{a_2+a_4+a_6}{2} \ \ {\rm  and} \ \  T_4=\frac{a_3+a_4+a_5}{2},$$
 
 $$Q_1=\frac{a_1+a_2+a_4+a_5}{2}, \ \ Q_2=\frac{a_1+a_3+a_4+a_6}{2} \ \ {\rm and} \ \  Q_3=\frac{a_2+a_3+a_5+a_6}{2}.$$

\begin{definition} { \rm An admissible coloring of a tetrahedron $\Delta$ is an assignment of an admissible $6$-tuple  $(a_1, a_2, a_3, a_4, a_5, a_6)$ of elements of $I_r$ to the edges of $\Delta$ so that the three numbers assigned to
 the edges of each face form an admissible triple. In this setting, the quantities $T_i$ and $Q_i$ defined above correspond  to the sums of colorings over faces of the tetrahedron, and   the sums of colorings of  edges of normal quadrilaterals in $\Delta$.}
 \end{definition}
 Given a compact orientable 3-manifold $M$ consider a triangulation $\tau$ of $M$. If $\partial M\neq \emptyset$  we will allow $\tau$ to be  a (partially) ideal triangulation,
 where some of the vertices of the tetrahedra are truncated and the truncated faces triangulate $\partial M$. Given a partially ideal triangulation $\tau$ the set $V$ of interior vertices of $\tau$ is the set of vertices of $\tau$ which do not lie on $\partial M.$ 
 Also we write $E$ for the set of interior edges (thus excluding edges coming from the truncation of vertices).
  A \emph{coloring at level $r$} of the triangulated $3$-manifold $(M,\mathcal \tau)$ is an assignment of elements of $I_r$ to the edges of $\mathcal \tau$ and is \emph{admissible} if the $6$-tuple assigned to the edges of each tetrahedron of $\tau$ is admissible. Let $c$ be an admissible coloring of $(M,\tau)$ at level $r.$ Given a  coloring  $c$ and an edge $e \in E$
  let  $|e|_c=(-1)^{c(e)}[c(e)+1]$. Also  for $\Delta$ a tetrahedron in $\tau$ let $|\Delta|_c$ be the quantum $6j$-symbol  corresponding to the admissible $6$-tuple assigned to $\Delta$ by $c$.
    Finally, let $A_r(\tau)$ denote the set of $r$-admissible colorings of $\tau$  and let
   $\displaystyle{\eta_r=\frac{2\sin(\frac{2\pi}{r})}{\sqrt{r}}}$. 
   
   We are now ready to define the  Turaev-Viro invariants as a state-sum over $A_r(\tau)$:
   
\begin{theorem}{\rm (\cite{KauLins, TuraevViro})}\label{TVdef} Let $M$ be a compact orientable  manifold closed or with boundary.  Let $b_2$ denote the second ${\Z}_2$-Betti number of $M$
and 
$b_0$   is the number of closed  connected components in $M$.
Then the state sum  
 \begin{equation}
 TV_r(M)=2^{b_2-b_0} \ \eta_r^{2|V|}\sum_{c\in A_r(\tau)}\prod_{e\in E}|e|_c\prod_{\Delta\in \tau}|\Delta|_c,
 \label{sumTV}
 \end{equation}
is independent of the partially ideal triangulation $\tau$ of $M,$ and thus defines a topological invariant of $M.$

\end{theorem}

Note that while the definition above differs slightly from   \cite[Definition 7]{KauLins} by the use of even colors only,  by \cite[Theorem 2.9]{DKY} the two definitions are essentially the same; they only differ by the factor of $2^{b_2-b_0}.$ 
 The restriction of the coloring set to only even integers is reminiscent to
 the $SO(3)$ quantum invariant theory and it facilitates the study of the Turaev-Viro invariants, for odd levels, via the $SO(3)$-TQFT theory of
 \cite{BHMV2}.

\subsection{Witten-Reshetikhin-Turaev invariants and TQFT}
\label{sec:TQFTdef}
The definition of the Turaev-Viro invariants given above will be useful for us to show that the upper limit $LTV(M)$ is well defined (i.e. it is finite). However, in order to understand
the topological properties of $LTV(M)$ (i.e. its behavior under prime and toroidal decompositions of 3-manifolds) it will be convenient for us to view Turaev-Viro invariants through their relation
to the Witten-Reshetikhin-Turaev invariants $RT_r(M)$  (\cite{ReTu}), and the Topological Quantum Field Theories (TQFTs) they are part of. 

The Witten-Reshetikhin-Turaev TQFTs are functors from the category of cobordisms in dimension $2+1$ to the category of finite dimensional vector spaces; they associate a finite dimensional $\mathbb{C}$-vector space $RT_r(\Sigma)$ to each compact oriented surface $\Sigma$. Moreover, their values on closed $3$-manifolds $M$ are the Witten-Reshetikhin-Turaev invariants $RT_r(M)$ which are $\mathbb{C}$-valued and are related to surgery presentations of $3$-manifolds and colored Jones polynomials.
Also, for $M$ with boundary $\partial M=\Sigma,$ $RT_r(M) \in RT_r(\Sigma)$ is a vector.
\\ We will introduce these TQFTs in the skein-theoretic framework of Blanchet, Habegger, Masbaum and Vogel \cite{BHMV2}. As we restrict to level $r$ odd, the resulting TQFTs are the so called $SO(3)$-TQFTs. Below we will sketch only the properties of these TQFTs we need, referring the reader to \cite{BHMV2} for a precise definition and more details.

 To fix some notations, we recall that when $V$ is a $\mathbb{C}$-vector space, $\overline{V}$ denotes the $\mathbb{C}$-vector space that is $V$ as an abelian group and whose scalar multiplication is $\alpha \cdot v= \overline{\alpha}v$. When $V$ is a $\mathbb{C}$-vector space, an Hermitian form $\langle \cdot ,\cdot \rangle$ on $V$ is a map
$$\langle \cdot ,\cdot \rangle : V \otimes V \rightarrow \mathbb{C},$$
that satisfies $\langle \alpha v+w, v'\rangle=\alpha \langle v,v'\rangle +\langle w,v' \rangle$ and $\langle w,v \rangle=\overline{\langle v,w\rangle}$. 
Note that an Hermitian form can be considered a bilinear form over $V \otimes \overline{V}$.

\begin{remark} {\rm We note that the invariants $RT_r(M)$ in \cite{BHMV2} are only well-defined up to a $2r$-th root of unity, this ambiguity being called the anomaly of the TQFT. Resolving the anomaly requires considering $3$-manifolds $M$ with an additional structure called a $p_1$-structure, see \cite{BHMV2} for details. Since in this article we will only be interested in the moduli of the invariants  $RT_r(M)$, we will neglect the anomaly. We warn the reader, however,  that the rules for computing $RT_r$ in Theorem \ref{TQFT} below have to be understood to hold up to a root of unity.}
\end{remark}

We summarize the main properties of the $\mathrm{SO}(3)$-TQFT defined in \cite{BHMV2} in the following theorem:
\begin{theorem}\label{TQFT} {\rm (\cite[Theorem 1.4]{BHMV2})}
Let $r$ be an odd integer and $A$ be a primitive $2r$-th root of unity. Then there exists a TQFT functor $RT_r$ in dimension 2+1 satisfying:
\begin{enumerate}
\item For any oriented compact closed $3$-manifold $M,$ $RT_r(M)\in \mathbb{C}$ is a topological invariant. Moreover if $\overline{M}$ is the manifold $M$ with the opposite orientation, then $RT_r(\overline{M})=\overline{RT_r(M)}$.
\item We have $RT_r(S^2\times S^1)=1$ and $RT_r(S^3)=\eta_r={\displaystyle{ \frac{2\sin(\frac{2\pi}{r})}{\sqrt{r}}}}$.
\item The invariants $RT_r$ are multiplicative under disjoint union of $3$-manifolds, and for connected sums we have
 $$RT_r(M\# M')=\eta_r^{-1}RT_r(M)RT_r(M').$$ 
\item For any closed compact oriented surface $\Sigma,$ $RT_r(\Sigma)=V_r(\Sigma)$ is a finite dimensional $\mathbb{C}$-vector space and for disjoint unions we have natural isomorphisms $$V_r(\Sigma_1 \coprod \Sigma_2) \simeq V_r(\Sigma_1)\otimes V_r(\Sigma_2).$$
Moreover, $V_r(\emptyset)=\mathbb{C}$ and for any oriented surface $V_r(\overline{\Sigma})$ is the $\mathbb{C}$-vector space $\overline{V_r(\Sigma)}$.

\item To every compact, oriented $3$-manifold $M$ with a fixed homeomorphism $\partial M\simeq\Sigma$ there is an associated vector $RT_r(M)\in V_r(\Sigma)$. 

Moreover for a disjoint union $M=M_1 \coprod M_2,$ we have
$$RT_r(M)=RT_r(M_1)\otimes RT_r(M_2) \in V_r(\Sigma_1)\otimes V_r(\Sigma_2).$$

\item For any odd integer $r,$ there is a natural Hermitian form
$$\langle \cdot , \cdot \rangle : V_r(\Sigma) \otimes V_r(\overline{\Sigma}) \rightarrow \mathbb{C},$$
with the following property:  Given  $M$ a compact oriented $3$-manifold and $\Sigma$ an embedded surface in $M,$ if  we let $M'$ be the manifold obtained by cutting $M$ along $\Sigma,$ 
with  $\partial M'=\Sigma \coprod \overline{\Sigma} \coprod \partial M$, then
 we have $RT_r(M)=\Phi (RT_r(M'))$. Here $\Phi$ is the linear map

$$ \Phi :  V_r(\Sigma) \otimes V_r(\overline{\Sigma}) \otimes V_r(\partial M)  \longrightarrow  V_r(\partial M),$$

defined  by $ \Phi (v \otimes w \otimes \varphi )=   \langle v , w \rangle \varphi$.

\end{enumerate}
\end{theorem}
\vskip 0.03in 

\noindent{\bf Mapping Cylinders.} A class of $3$-manifolds with boundary to which the construction can be applied are the mapping cylinders of maps of surfaces:
Given a surface $\Sigma$ and an element $\varphi \in \mathrm{Mod}(\Sigma)$ in its mapping class group, let
$$M_{\varphi}=[0,1]\times \Sigma \underset{(x,1)\sim \varphi(x)}{\cup} \Sigma.$$
Then $RT_r(M_{\varphi})$ is a vector in $V_r(\Sigma)\otimes \overline{V_r(\Sigma)}$. The later space can be identified with $\mathrm{End}(V_r(\Sigma))$ as $V_r(\overline{\Sigma})\simeq V_r( \Sigma)^*$ by the natural Hermitian form.  The assignment  $\rho_r(\varphi)=RT_r(M_{\varphi})$, defines a projective representation
$$ \rho_r :   \mathrm{Mod}(\Sigma) \longrightarrow  \mathbb{P}\mathrm{End}(V_r(\Sigma)). $$
These representations are known as the  the quantum representations of mapping class groups; they are projective because of the above mentioned TQFT anomaly factor.
 Given the mapping torus 
$$N_{\varphi}=[0,1]\times \Sigma /_{(x,1)\sim (\varphi(x),0)},$$ 
of a class  $\varphi \in \mathrm{Mod}(\Sigma)$, by  \cite[Formula 1.2]{BHMV2} we have $RT_r(N_{\varphi})=\mathrm{Tr}(\rho_r (\varphi))$.

We will need the following well-known fact which we state as a lemma. A proof can be found for example in \cite{FreedKrush}.
\begin{lemma}\label{ellipticinvol}Let $T^2$ be the 2-dimensional torus and let $\varphi :  T^2\simeq S^1 \times S^1  \rightarrow  S^1 \times S^1,$
be the elliptic involution, defined by $(x,y) \rightarrow  (-x,-y)$. Then $\rho_r(\varphi)=\mathrm{id}_{V_r(T^2)}.$
\end{lemma}

\smallskip

In the next statement we summarize from \cite{BHMV2} the facts 
about the dimensions of  $V_r(\Sigma)$ that we will need.

\begin{theorem}{\rm (\cite{BHMV2})}  \label{TQFTbasis} \ We have the following:
\begin{enumerate}
\item For any odd integer $r\geqslant 3,$ and any primitive $2r$-th root of unity, the vector space  $V_r(T^2)$ has dimension $\displaystyle{\frac{r-1}{2}}$  and 
 the Hermitian form $\langle \cdot , \cdot \rangle$ on $V_r(T^2)$ is definite positive.

\item If $\Sigma_g$ is the closed compact oriented surface of genus $g \geqslant 2$,  then $\mathrm{dim}(V_r(\Sigma_g))$ is a polynomial in $r$ of degree $3g-3$. 
\end{enumerate}
\end{theorem}
\begin{proof} The first assertion of part (1) is proved in \cite[Corollary 4.10]{BHMV2} and the second assertion is given in \cite[Remark 4.12]{BHMV2}. The second assertion follows by 
\cite[Corollary 1.16 and Remark(iii)]{BHMV2}.
\end{proof}

To continue, recall that the double $D(M)$ of a manifold $M$ is defined as $M \coprod \overline{M}$ if $M$ is closed and as $M\underset{\Sigma}{\cup} \overline{M}$ if $M$ has non-empty boundary. 
We end this section with a theorem that for a manifold $M$ relates the $SO(3)$-Turaev-Viro invariants $TV_r(M)$ defined in Section \ref{sec:TVdef}, to the $RT_r$ invariant of the double $D(M)$ of $M$.

 \begin{theorem}[]\label{RTdoubleTV}{\rm (\cite{BePe})}\ 
Let $M$ be a $3$-manifold with boundary, $r$ be an odd integer and $q=e^{\frac{2i\pi}{r}}$. Then  
$$TV_r(M,q)=\eta_{r}^{-\chi(M)}RT_r(D(M),e^{\frac{i\pi}{r}}),$$
where $\chi(M)$ is the Euler characteristic of $M$ and $\eta_r=\displaystyle{\frac{2\sin(\frac{2\pi}{r})}{\sqrt{r}}=RT_r(S^3)}.$
\end{theorem}
The proof of Theorem \ref{RTdoubleTV}  for closed manifolds is due to Roberts \cite{Roberts} and  Walker and Turaev \cite{Turaevbook}.
The proof for manifolds with non-empty boundary is essentially due to
Benedetti and Petronio  \cite{BePe}: Although in \cite{BePe} only  the invariants $RT_r(M)$   corresponding to the $\mathrm{SU}(2)$-TQFT are considered,  the proof can be adapted 
is the $SO(3)$-TQFT setting. This was done in \cite[Theorem 3.1]{DKY}.

%%%%%%%%%%%%%%

\section{Finiteness of $LTV$} 
\label{sec:LTVfinite}
The goal of this section is to prove Theorem \ref{nbtetrahedrabound}.
We first provide an upper bound to the quantum $6j$-symbols at level $r$ and at the root $q=e^{\frac{2i\pi}{r}}$. Using this we bound the invariants $TV_r(M)$ of a $3$-manifold $M$ in terms of the number of tetrahedra in a triangulation of $M.$ This, in particular, will  prove that $LTV(M)$ is finite.

\subsection{Quantum factorials and the Lobachevsky function}
It is a well known fact in quantum topology that the asymptotics of quantum factorials are related to the Lobachevsky function. In this section we give a version of this fact at the root $q=e^{\frac{2i\pi}{r}}.$
\\For $P$ a Laurent polynomial, let $ev_r(P)$ be the evaluation of the absolute value of $P$ at $q=e^{\frac{2i\pi}{r}}$, that is 
$$ev_r(P)=|P(e^{\frac{2i\pi}{r}})|.$$
Let also $\Lambda(x)$ denote the Lobachevski function, defined by 
$$\Lambda(x)=-\int_0^x  \log  |2 \sin(x)|dx.$$
We estimate the growth of quantum factorials at $q=e^{\frac{2i\pi}{r}}$ using the following lemma.
\begin{proposition}\label{quantumfact} Given an integer   $0<n<r$ we have $$\log(ev_r (\lbrace n \rbrace ! ))=-\frac{r}{2\pi}\Lambda(\frac{2 n \pi}{r}) +O(\log r).$$
Moreover, in this estimate $O(\log r)$ is uniform: there exists a constant $C_3$ independent of $n$ and $r,$ such that $O(\log r)\leqslant C_3  \log r$. 
\end{proposition} 
\begin{proof}
 First note that $ev_r (\lbrace n \rbrace ! )=\underset{j=1}{\overset{n}{\prod}} 2 \sin (\frac{2 j \pi}{r})$.  In this product, as $r$ is odd and $0<n<r,$ all factors are non-zero. Thus we can write
$$\log (ev_r (\lbrace n \rbrace ! ))=\underset{j=1}{\overset{n}{\sum}} \log |2 \sin (\frac{2 j \pi}{r})|.$$
The function  $$f(t)=\log |2 \sin(\frac{2\pi t}{r})|,$$ is differentiable on $(0,\frac{r}{2})$ and on $(\frac{r}{2},r)$. 
\vskip 0.05in

\noindent  {\textit{Case 1:} } Assume that $n<\frac{r}{2}$. The Euler-Mac Laurin formula gives  that 
$$\log (ev_r (\lbrace n \rbrace ! ))=\int_1^n \log |2 \sin (\frac{2\pi t}{r})|dt +\frac{f(1)+f(n)}{2}+R_0,$$
where the $R_0\leqslant \frac{1}{2}\int_1^n |f'(t)|dt\leqslant 2 \underset{t \in [1,n]}{\mathrm{sup}}|f(t)| $ as $f$ is increasing then decreasing on $(0,\frac{r}{2})$.

 Note that $|n -\frac{r}{2}|>\frac{1}{2}$ and $\sin(t)\geqslant \frac{2t}{\pi}$ for $0\leqslant t \leqslant \frac{\pi}{2}.$
 Hence, the quantities $|f(1)|,|f(n)|$ and $\underset{t \in [1,n]}{\mathrm{sup}}|f(t)|$ are all bounded by $|\log(\frac{4}{r})|$ for big $r$.
Moreover,  for big $r$, we have
$$\left|\int_0^1 \log |2 \sin (\frac{2\pi t}{r})|dt\right| \leqslant \left|\int_0^1 \log(\frac{4t}{r})dt\right|\leqslant |\log r|+\left|\int_0^1 \log(4t)dt \right|.$$ 
Thus
\begin{multline*}\log (ev_r (\lbrace n \rbrace ! ))=\int_0^{n} \log |2 \sin (\frac{2\pi t}{r})|dt +O(\log r)
\\=\frac{r}{k\pi}\int_0^{\frac{2 n \pi}{r}} \log |2 \sin t| dt +O(\log r)=-\frac{r}{2\pi}\Lambda(\frac{2 n\pi}{r}) +O(\log r).
\end{multline*}

Note that in all our estimations the $O(\log r)$ was independent on $0<n<\frac{r}{2}.$
\vskip 0.09in

\noindent {\textit{Case 2:}}
Assume that  $n\geqslant \frac{r}{2}$. Now we write $\log (ev_r (\lbrace n \rbrace ! ))$ as a sum of two terms
$$\log (ev_r (\lbrace n \rbrace ! ))=\underset{j=1}{\overset{\lfloor \frac{r}{2}\rfloor }{\sum}} \log |2 \sin (\frac{2 j \pi}{r})|+\underset{j=\lceil \frac{r}{2}\rceil}{\overset{n }{\sum}} \log |2 \sin (\frac{2 j \pi}{r})|.$$
Applying the Euler-Mac Laurin formula for each sum we get
\begin{multline*}\log (ev_r (\lbrace n \rbrace ! ))=\int_1^{\lfloor \frac{r}{2}\rfloor} \log |2 \sin (\frac{2\pi t}{r})| dt + \int_{\lceil \frac{r}{2}\rceil}^n \log |2 \sin (\frac{2\pi t}{r})| dt 
\\ +\frac{f(1)+f(\lfloor \frac{r}{2}\rfloor)+f(\lceil \frac{r}{2}\rceil)+f(n)}{2}+R_0,
\end{multline*}

where 
$$R_0\leqslant \frac{1}{2}\left(\int_1^{\lfloor \frac{r}{2}\rfloor} |f'(t)|dt+\int_{\lceil \frac{r}{2}\rceil}^n |f'(t)|dt \right) \leqslant 2 \left(\underset{t \in [1,\lfloor \frac{r}{2}\rfloor]}{\mathrm{sup}}|f(t)|+\underset{t \in [\lceil \frac{r}{2}\rceil,n]}{\mathrm{sup}}|f(t)|\right),$$
 as $f$ is increasing then decreasing on $(0,\frac{r}{2})$ and increasing then decreasing on $(\frac{r}{2},r).$ 
Since $r$ is odd and $\frac{r}{2}$ is a half integer, similarly as in Case 1,  we have that $f(1),f(\lfloor \frac{r}{2}\rfloor),f(\lceil \frac{r}{2}\rceil),f(n)$ are all $\leqslant C_3 \log r$ for some $C_3$ independent of $n$.  Also 
\begin{multline*}\int_1^{\lfloor \frac{r}{2}\rfloor} \log |2 \sin (\frac{2\pi t}{r})| dt + \int_{\lceil \frac{r}{2}\rceil}^n \log |2 \sin (\frac{2\pi t}{r})| dt
\\ =\int_0^n  \log |2 \sin (\frac{2\pi t}{r})| dt +O(\log r)=-\frac{r}{2\pi}\Lambda(\frac{2 n\pi}{r})+O(\log r).
\end{multline*}
This concludes the proof of Proposition \ref{quantumfact}.
\end{proof}

%%%%%%%%%%%%%%%%

\subsection{Upper bounds for quantum $6j$-symbols} 
\label{sec:6jbound}
In this section, we find an upper bound for the quantum $6j$-symbols that we will need for the proof of Theorem \ref{nbtetrahedrabound}. We show the following:

\begin{proposition}\label{6jbound}For any $r$-admissible $6$-tuple $(a,b,c,d,e,f),$ we have that
$$\frac{2\pi}{r}\log \left( ev_r \begin{vmatrix}
a_1  & a_2 & a_3 \\ a_4  & a_5 & a_6 
\end{vmatrix}\ \right)\leqslant v_8+8\Lambda(\frac{\pi}{8})+O(\frac{\log r}{r}).$$
Moreover, in this estimate the $\displaystyle{O(\frac{\log r}{r})}$ is uniform: there exists a constant $C_3$ independent of the $a_i's$ and $r,$ such that $\displaystyle{O(\frac{\log r}{r})\leqslant C_3  \frac{\log r}{r}}$. 
\end{proposition}
\begin{remark} The bound of $v_8+8\Lambda(\frac{\pi}{8})$  given above is not optimal. In a closely related context, Costantino proved that the growth rates of quantum $6j$-symbols with so-called hyperbolic admissibility conditions are given by volumes of hyperbolic truncated tetrahedra \cite{costantino:6j}. The argument of Costantino is also applicable in our context from some (but not all) values of the $a_i$'s.  More recently, in \cite{BDKY} the authors with Belletti and Yang, building on the analysis given in this article,
proved that  the upper bound of Proposition \ref{6jbound} can indeed be replaced
with the maximum volume of a hyperbolic truncated tetrahedron, which is $v_8.$
\end{remark}
\begin{proof}
By Equation \ref{6j}, we have

\begin{equation*}\log \left(ev_r \begin{vmatrix}
a_1  & a_2 & a_3 \\ a_4  & a_5 & a_6 
\end{vmatrix} \ \right)= {\sum_{i=1}^4 \log |\Delta(F_i)}|+\log |S|+O(\log r),
\end{equation*}

where $$S= \sum_{z=\max \{T_1, T_2, T_3, T_4\}}^{\min\{ Q_1,Q_2,Q_3\}}\frac{(-1)^z\lbrace  z+1 \rbrace !}{\prod_{j=1}^4\lbrace z-T_j\rbrace !\prod_{k=1}^3\lbrace Q_k-z\rbrace !},$$
and $F_1, F_2, F_3, F_4$ are defined in equation (\ref{notation}).
Each of the $\Delta(F_i)$ terms is a product of 4 quantum factorials. Moreover as $Q_j-T_i\leqslant (r-2)$ for all $i$ and $j,$ we have
\begin{equation*}log|S| \leqslant
\\ \underset{\mathrm{max} T_i \leqslant z \leqslant \mathrm{min} Q_j}{\mathrm{max}}  \log \left| \frac{\lbrace  z \rbrace !}{\prod_{j=1}^4\lbrace z-T_j\rbrace !\prod_{k=1}^3\lbrace Q_k-z\rbrace !}\right|
+O(\log r).
\end{equation*}

Here we used the fact that  $S$ is a sum of a polynomial number of terms and that for $0\leqslant z\leqslant r-2,$ we have $\log|\lbrace z+1 \rbrace|\leqslant O(\log r)$ for some $O(\log r)$ independent on $z$.
\\  To estimate these terms we will use  Proposition \ref{quantumfact}. We will write
 $$A_i=\frac{2\pi a_i}{r},\ U_j=\frac{2\pi T_j}{r}, \ {\rm and} \ V_k=\frac{2\pi Q_k}{r}.$$ We have
\begin{multline}\label{eq:6jasymptotics}log \left(ev_r \begin{vmatrix}
a_1  & a_2 & a_3 \\ a_4  & a_5 & a_6 
\end{vmatrix} \ \right)\\ \leqslant\frac{r}{2\pi}\left(v(\frac{A_1}{2},\frac{A_2}{2},\frac{A_3}{2})+v(\frac{A_1}{2},\frac{A_5}{2},\frac{A_6}{2})+v(\frac{A_2}{2},\frac{A_4}{2},\frac{A_6}{2})+v(\frac{A_3}{2},\frac{A_4}{2},\frac{A_5}{2})\right)
\\ + \frac{r}{2\pi}\ \underset{Z}{\mathrm{max}} \  g(Z,A_i) +O(\log r),
\end{multline}
where 
$$v(\alpha, \beta, \gamma)=\frac{1}{2}\left(\Lambda \left( \alpha+\beta+ \gamma \right)-\Lambda \left( \beta+ \gamma-\alpha \right)-\Lambda \left( \alpha+\gamma-\beta \right)-\Lambda \left( \alpha+\beta- \gamma  \right) \right),$$
and 
\begin{equation*}g(Z,A_i)=\sum_{i=1}^4 \Lambda(Z-U_i)+\sum_{j=1}^3 \Lambda(V_i-Z)-\Lambda(Z).
\end{equation*}
Since the  function $\Lambda$ is bounded, the functions $v$ and $g$ are also bounded. Thus
$$\log \left(ev_r \begin{vmatrix}
a_1  & a_2 & a_3 \\ a_4  & a_5 & a_6 
\end{vmatrix} \ \right)\leqslant \frac{r}{2\pi}C_1+O(\log r),$$ for some constant $C_1$.
We show that one can use $C_1=v_8+8\Lambda \left(\frac{\pi}{8}\right)$ by computing the maximum of functions $v$ and $g$ above. This is done using the following lemma that
we will prove in the Appendix.
\begin{lemma}\label{maxfunct6j} The maximum of the function $v$ is $\displaystyle{\frac{v_8}{4}}$ and the maximum of the function $g$ is $\displaystyle{8\Lambda \left( \frac{\pi}{8}\right)}$.
\end{lemma}
\end{proof}
\subsection{$LTV$ and state-sums}
\label{sec:LTVstatesum}
Using state sums for the Turaev-Viro invariants and Proposition \ref{6jbound} we prove Theorem \ref{nbtetrahedrabound} stated in the Introduction.

\begin{named}{Theorem \ref{nbtetrahedrabound}}Suppose that $M$ is a compact, oriented manifold with a triangulation consisting of  $t$ tetrahedra. Then, we have
$$LTV(M)\leqslant  2.08\  v_8 \ t,$$
where $v_8\simeq  3.6638$ is the volume of  a regular ideal octahedron.
\end{named}
\begin{proof}
Let $\tau$ be a triangulation of $M$ with $t$ tetrahedra. Recall that by   equation (\ref{sumTV}) in the statement of Theorem 
 \ref{TVdef}
$$TV_r(M)=2^{b_2-b_0} \ \eta_r^{2|V|}\sum_{c\in A_r(\tau)}\prod_{e\in E}|e|_c\prod_{\Delta\in \tau}|\Delta|_c.$$

Since the term $2^{b_2-b_0} $ is independent of $r$ we may ignore it.
Recall that  $A_r(\tau)$ is the set of admissible $r$-colorings of the edges of the triangulation $\tau$.
The number of elements of the set $A_r(\tau)$ is bounded by a polynomial in $r$ as each edge must be colored by an element of $\lbrace 0, 2 ,\ldots r-3\rbrace$. So 
$$\frac{2\pi}{r}\log |TV_r(M)|\leqslant \frac{2\pi}{r} \log \left( |A_r(\tau)| |P| \right) \leqslant \frac{2\pi}{r}\log |P|+O(\frac{\log r}{r}),$$ 
 where $P$ is the term in the sum of maximal $\log.$
Moreover, 
$$\log |\eta_r|=\log \left|\frac{2\sin(\frac{2\pi}{r})}{\sqrt{r}}\right|=O(\log r),$$ 
and for the factors $|e|_c$ we have:
$$\log |e|_c=\log \left|\frac{\sin(\frac{2\pi (c(e)+1)}{r})}{\sin(\frac{2\pi}{r})}\right|=O(\log r).$$ 
Finally, by Proposition \ref{6jbound}, for any $c \in A_r(\tau)$ and any $\Delta \in \tau,$ the factor $\frac{2\pi}{r}\log |\Delta|_c$ is bounded above by $C_1=v_8+8 \Lambda \left( \frac{\pi}{8}\right)\simeq 7.5914< 7.6207 \simeq 2.08 v_8$ and $P$ has $t$ such factors. The theorem follows.
\end{proof}

%%%%%%%%%%%%%%%%%%%%%%%%%%%

\section{Cutting along tori}
\label{sec:cuttingtori}
In this section, we will prove a theorem (Theorem \ref{toruscutting} below) that describes the behavior of the Turaev-Viro invariants when cutting a 3-manifold along a torus. This will follow from the TQFT properties of $TV_r$ and Cauchy-Schwarz type inequalities. We will need the following lemma.
\begin{lemma}\label{Cauchyschwarz}Let $v_1,v_2,\ldots ,v_n$ be vectors in a $\mathbb{C}$-vector space $V$ with positive definite Hermitian form $\langle \cdot , \cdot \rangle$ and norm $|| \cdot ||$. Then we have
$$||\underset{i=1}{\overset{n}{\sum}} v_i ||^2 \leqslant n \underset{i=1}{\overset{n}{\sum}} ||v_i||^2.$$
\end{lemma}
\begin{proof}
Let $e_1,\ldots ,e_d$ be an orthonormal basis of $\mathrm{Span}(v_1,\ldots ,v_n)$ and write
$$v_i=\underset{j=1}{\overset{d}{\sum}} \lambda_{ij}e_j.$$
Then 
$$||\underset{i=1}{\overset{n}{\sum}} v_i ||^2=\underset{j=1}{\overset{d}{\sum}} |\underset{i=1}{\overset{n}{\sum}}\lambda_{ij}|^2\leqslant \underset{j=1}{\overset{d}{\sum}}n \underset{i=1}{\overset{n}{\sum}}|\lambda_{ij}|^2=n  \underset{i=1}{\overset{n}{\sum}} ||v_i||^2,$$
where the inequality follows from the Cauchy-Schwarz inequality in $\mathbb{C}^n.$
\end{proof}
\begin{theorem} \label{toruscutting}
Let $r \geqslant 3$ be an odd integer and let $M$ be a  compact oriented 3-manifold with empty or toroidal  boundary. Let   $T \subset M$ be  an embedded torus  and let $M'$ be the manifold obtained by cutting $M$ along $T.$ Then
$$TV_r(M)\leqslant \left(\frac{r-1}{2}\right) TV_r(M'),$$
and $$LTV(M) \leqslant LTV(M').$$
If moreover $T$ is separating then $TV_r(M)\leqslant TV_r(M')$.
\end{theorem} 

\begin{proof}  Let   $\Sigma=\partial M$.  We will distinguish two cases:
\vskip 0.05in
\noindent {\textit{Case 1:}} Suppose that the torus $T$ is non-separating. The torus $T$ inherits an orientation from that of $M$.
The manifold $M'$, obtained by cutting $M$ along $T$, has boundary  $\partial M'=\Sigma \coprod T \coprod \overline{T}$, where $\overline{T}$  is the torus $T$ with opposite orientation. 
With the notation of  Theorem \ref{TQFT}, we have
$$RT_r(M') \in V_r(\partial M')=V_r(\Sigma)\otimes V_r(T) \otimes V_r(\overline{T}) \ \ {\rm and}
\ \  RT_r(M) \in V_r(\Sigma).$$ Furthermore $RT_r(M)=\Phi(RT_r(M'))$,
where $\Phi$ is the 
the contraction map of Theorem \ref{TQFT}(6): 
$$\begin{array}{rccl}\Phi: & V_r(\Sigma)\otimes V_r(T )\otimes V_r(\overline{T}) & \longrightarrow& V_r(\Sigma)
\\ & v\otimes w_1 \otimes w_2 & \longrightarrow & \Phi(v\otimes w_1 \otimes w_2)=\langle w_1, w_2 \rangle  v.
\end{array}$$

By Theorem \ref{RTdoubleTV}, we have $$TV_r(M')=\langle RT_r(M'), RT_r(M') \rangle=||RT_r(M')||^2\ \ {\rm and } \ \
TV_r(M)=||RT_r(M)||^2.$$

By hypothesis, $\Sigma$ is a (possibly empty) union of, say $n$,  tori; thus we have $V_r(\Sigma)=V_r(T^2)^{\otimes n}$
and $V_r(\partial M')=V_r(T^2)^{\otimes {(n+2)}}$.
 By Theorem \ref{TQFTbasis}  the Hermitian form on  $V_r(T^2)$  is definite positive. Hence, setting   $m=\frac{r-1}{2}$,  we have an  orthonormal basis 
 $\left( \varphi_i \right)_{1\leqslant j \leqslant m}$ on $V_r(T^2)$.
Using this basis we can write
 $$RT_r(M')=\underset{1\leqslant i,j \leqslant m}{\sum} v_{ij}\otimes \varphi_i \otimes \varphi_j,$$
where the $v_{ij}$ are vectors in $V_r(\Sigma)$ which is also an Hermitian vector space with definite positive Hermitian form.
 We have that 
\begin{equation}\label{eq:TV-M'}TV_r(M')=||RT_r(M')||^2=\underset{1\leqslant i,j \leqslant m}{\sum} ||v_{ij}||^2
\end{equation}
as $\varphi_i \otimes \varphi_j$ is an orthonormal basis of $V_r(T) \otimes V_r({\overline {T}})$.
On the other hand, applying the contraction map $\Phi,$ we get:
$$RT_r(M)=\Phi(RT_r(M'))=\underset{1\leqslant i,j \leqslant m}{\sum} \langle \varphi_i,\varphi_j \rangle v_{ij}=\underset{1\leqslant i \leqslant m}{\sum} v_{ii},$$
as $\varphi_i$ is an orthonormal basis of $V_r(T)$. Thus 
$$TV_r(M)=||RT_r(M)||^2=\left| \left| \underset{1\leqslant i \leqslant m}{\sum} v_{ii} \right| \right|^2 \leqslant m \underset{1\leqslant i \leqslant m}{\sum} ||v_{ii}||^2 \leqslant m \underset{1\leqslant i,j \leqslant m}{\sum} ||v_{ij}||^2 = m TV_r(M'),$$
where the first inequality follows from Lemma \ref{Cauchyschwarz} and the last equality from Equation (\ref{eq:TV-M'}). This proves the first part of Theorem \ref{toruscutting}.
\vskip 0.08in

\noindent  { \textit{Case 2:}} Let $T$ be a separating torus and let $M_1$ and $M_2$ be the two components of $M\setminus T.$ Let us write $\partial M_1=T\cup \Sigma_1$ and $\partial M_2=\overline{T}\cup \Sigma_2,$ where $\Sigma_1$ and $\Sigma_2$ are actually (possibly empty) unions of tori.  By Theorem \ref{TQFTbasis} the natural Hermitian form on  $V_r(\Sigma_1)$ and $V_r(\Sigma_2),$ 
are  positive definite.
Hence we have orthonormal bases  $\left( \varphi_i \right)_{ i}$ and $\left( \psi_j \right)_{j}$ of $V_r(\Sigma_1)$ and $V_r(\Sigma_2),$  respectively.
We can write:
$$RT_r(M_1)=\sum_i v_i \otimes \varphi_i,$$
where the $v_i$ are vectors in $V_r(T)$ and 
$$RT_r(M_2)=\sum_j w_j \otimes \psi_j,$$
where the $w_j$ are vectors in $V_r(\overline{T})$.
From this we get $$TV_r(M_1)=||RT_r(M_1)||^2=\sum_i ||v_i||^2,$$ and likewise $$TV_r(M_2)=||RT_r(M_2)||^2=\sum_j ||w_j||^2.$$
On the other hand, one has 
\begin{multline*}TV_r(M)=||RT_r(M)||^2=||\sum_{i,j} \langle v_i, w_j \rangle \varphi_i \otimes \psi_j ||^2=\sum_{i,j} |\langle v_i, w_j \rangle |^2
\\ \leqslant \sum_{i,j} ||v_i||^2 ||w_j||^2=TV_r(M_1)TV_r(M_2),
\end{multline*}
by the Cauchy-Schwarz inequality.
\end{proof}
\vskip 0.04in

As a special case of this theorem we get the following.
\begin{corollary}\label{dehnfilling}Let $M'$ be a compact oriented 3-manifold with non-empty  toroidal boundary and let $M$ be a manifold obtained from $M'$ by Dehn-filling some of the
boundary  components.
Then
$$TV_r(M)\leqslant TV_r(M'),$$
and thus $$LTV(M)\leqslant LTV(M').$$
\end{corollary}
\begin{proof} Suppose we are  Dehn filling $n$ components of $\partial M'.$
The Dehn filling $M$ is obtained from $M'\coprod_{i=1}^n V_i$ by gluing the boundary of each solid torus $V_i=D^2\times S^1$  to a boundary component of $M'$.
We can do the Dehn filling one component at a time.
Thus Corollary \ref{dehnfilling} is an immediate consequence of Theorem \ref{toruscutting}, the multiplicativity of $TV_r$ under disjoint union and the fact that 
$$TV_r(D^2\times S^1)=RT_r(S^2\times S^1)=1.$$
\end{proof}

%%%%%%%%%%%%%%%%%%%%%

\section{Bounds for  Seifert manifolds}
\label{sec:TVSeifert}
The previous section showed that the Turaev-Viro invariants are well behaved with respect to 3-manifold decompositions along tori. In this section  we deal with large $r$ asymptotic behavior
of Turaev-Viro invariants for Seifert manifolds.
We will use Corollary \ref{dehnfilling} to show the $TV_r$ invariants of Seifert manifolds are at most polynomially growing.
Our argument will be slightly different depending on whether the Seifert manifold has orientable or non-orientable base. The following lemma will help us reduce the latter to the former.

\begin{lemma}\label{nonorsurfacecut}Let $\tilde{\Sigma}$ be a  compact non-orientable surface. Then there is a simple closed curve $\gamma$ on $\tilde{\Sigma}$ that is orientation reversing such that the surface $\Sigma=\tilde{\Sigma} \setminus \gamma$ obtained from $\tilde{\Sigma}$ by cutting along $\gamma$ is orientable.
\end{lemma}

\begin{proof}
Without loss of generality, we can assume $\tilde{\Sigma}$ is closed. Otherwise, we can fill the boundary components by disks to get a closed surface $\tilde{\Sigma}',$ and a simple closed curve  for $\tilde{\Sigma}'$ that cuts it into an orientable surface.  Isotopying this curve away from the filled in disks on  $\tilde{\Sigma}'$ we will get a curve that satisfies the conclusion of the lemma for $\tilde{\Sigma}.$ 

Now, as closed non-orientable surfaces are characterized by their Euler characteristic which is at most $1,$ the surface $\tilde{\Sigma}$ is homeomorphic either to some $\mathbb{R}P^2 \# (\# T^2)^p$ or to some $K^2 \# (\# T^2)^p$ where $K^2$ is the Klein bottle and $p\geqslant 0.$ As $T^2$ is orientable, it will then be sufficient to find such a path in $\mathbb{R}P^2$ and $K^2.$
 In $\mathbb{R}P^2$ such a path is given by the  image of the equator  in $S^2$ by the cover map.
If we view the Klein bottle $K^2$ as
$$K^2\simeq S^1\times [0,1]_{/(x,1)\sim (-x,0)},$$
then the path $S^1\times\lbrace 0 \rbrace$ works.
\end{proof}

We are now ready to bound the Turaev-Viro invariants of Seifert fibered manifolds.
\begin{theorem}\label{TVSeifert}Let $M$ be a compact orientable manifold that is Seifert fibered.
Then there exist constants $A>0$ and $N>0$, depending on $M$, such that $$TV_r(M,e^{\frac{2i\pi}{r}})\leqslant Ar^N.$$ Thus we have $LTV(M)\leq 0$.
\end{theorem}
\begin{proof} We treat the case of orientable and non-orientable base separately. 
\vskip 0.08in

\noindent {\textit{Case 1:} } Let us first assume that the base of $M$ is orientable. Drilling out exceptional fibers and possibly one regular fiber, we obtain a Seifert manifold which is a  locally trivial fiber bundle over an oriented surface with non-empty boundary $\Sigma_{g,n}$. Here,  $n \geqslant 1$ is the number of boundary components of $\Sigma_{g,n}$ and $g$ is the genus.
Because $H^2(\Sigma_{g,n},\mathbb{Z})=0$, the Euler number of the fibration is zero  and the $S^1$-fibration is globally trivial. In the end, $M$ is a Dehn-filling of $\Sigma_{g,n}\times S^1$ for some oriented surface 
$\Sigma_{g,n}$ of genus $g$ and $n \geqslant 1$ boundary components.  

By Corollary \ref{dehnfilling}, we have $TV_r(M)\leqslant TV_r(\Sigma_{g,n}\times [0,1])$.
It remains to show that $TV_r(\Sigma_{g,n}\times S^1)$ is bounded by a polynomial. 
 By Theorem \ref{RTdoubleTV}, we have that
$$TV_r(\Sigma_{g,n}\times S^1)=RT_r(D(\Sigma_{g,n}\times S^1))=RT_r(D(\Sigma_{g,n})\times S^1).$$ 
The double surface $D(\Sigma_{g,n})$ is a closed orientable surface with $\chi(D(\Sigma_{g,n}))=2\chi(\Sigma_{g,n})=4-4g-2n;$ thus it is the surface $\Sigma_{2g+n-1}$ of genus $2g+n-1.$
So that we have
$$TV_r(\Sigma_{g,n}\times S^1)=RT_r(\Sigma_{2g+n-1}\times S^1)=\mathrm{Tr}(\rho_r(id_{\Sigma_{2g+n-1}}))=\mathrm{dim}(V_r(\Sigma_{2g+n-1})),$$
where $\rho_r$ is the quantum representation of $V_r(\Sigma_{2g+n-1})$. The last quantity is a polynomial by Theorem \ref{TQFTbasis}(2).
\vskip 0.08in

\noindent { \textit{Case 2:} } Assume that the base of $M$ is a non-orientable surface $\tilde{\Sigma}$. By Lemma \ref{nonorsurfacecut},  we have a simple closed curve $\gamma$ on $\tilde{\Sigma}$ such that $\Sigma=\tilde{\Sigma}\setminus \gamma$ is orientable. One can assume that $\gamma$ does not meet any exceptional fiber up to isotopying $\gamma$.  The fibers of $M$ corresponding to points on  $\gamma$ form an embedded Klein bottle $K^2$ in $M$. As $\gamma$ is orientation reversing and does not meet exceptional fibers, a regular neighborhood of $\gamma$ in $\tilde{\Sigma}$ will be a M\"obius band that does not meet the exceptional fibers, and its total space by the Seifert fibration will be homeomorphic to the twisted $I$-bundle over the $K^2$.
$$K^2\tilde{\times}I=[0,1]\times S^1 \times [-1,1]_{/(1,y,z)\sim (0,-y,-z)}.$$
The boundary of $K^2\tilde{\times}I $ is a separating torus in $M$ and, cutting $M$ along this torus, one obtains on one side a Seifert manifold $M'$ that fibers over the orientable surface $\Sigma$ and $K^2\tilde{\times}I $ on the other side. 

By Theorem \ref{toruscutting}, $TV_r(M)\leqslant TV_r(M') TV_r(K^2\tilde{\times}I)$.

We already know that $TV_r(M')$ is bounded by a polynomial so we only need to discuss  $TV_r(K^2\tilde{\times}I)$.
By Theorem \ref{RTdoubleTV} again, we have that $TV_r(K^2\tilde{\times}I)=RT_r(D(K^2\tilde{\times}I))$. The double of $K^2\tilde{\times}I$ is: 
$$D(K^2\tilde{\times}I)=[0,1]\times S^1 \times S^1_{/(1,y,z)\sim (0,-y,-z)}.$$
This is the mapping torus of the elliptic involution $s$ over $T^2$.
Thus we have that $$TV_r(K^2 \tilde{\times}I)=\mathrm{Tr}(\rho_r(s))=\mathrm{dim}(V_r(T^2))=\frac{r-1}{2},$$
as the elliptic involution is in the kernel of all quantum representations $\rho_r$ by Lemma \ref{ellipticinvol}.
\end{proof}

%%%%%%%%%%%%

\section{Turaev-Viro invariants and simplicial volume} 
\label{sec:LTVbound}
\subsection{A universal bound}
In this section we complete the proof of Theorem \ref{LTVbound} and deduce some corollaries. 
 First we note the following elementary properties of $LTV$ and $lTV$.

\begin{proposition}\label{LTVconnectedsum}  $LTV$ is subadditive under disjoint unions and connected sums of 3-manifolds while
$lTV$ is superadditive.

\end{proposition}
\begin{proof}
By  Theorems \ref{TQFT} and \ref{RTdoubleTV},  the Turaev-Viro invariants are multiplicative under  disjoint union.
Thus we have $TV_r(M\coprod M')=TV_r(M)TV_r(M')$ and
\begin{multline*}\underset{r \rightarrow \infty}{\limsup} \frac{2\pi}{r}\log |TV_r(M\coprod M')|=\underset{r \rightarrow \infty}{\limsup} \frac{2\pi}{r}\left( \log |TV_r(M)|+\log |TV_r(M') |\right)\\ \leqslant LTV(M)+LTV(M')
\end{multline*}
as the $\limsup$ operator is subadditive.
\\ For a connected sum, we have 
$$TV_r(M\# M')=\frac{TV_r(M)TV_r(M')}{TV_r(S^3)}=\eta_r^{-2}TV_r(M)TV_r(M').$$
But we have 
$$\frac{\log |\eta_r|}{r}=\frac{\log |\frac{2\sin(\frac{2\pi}{r})}{\sqrt{r}}|}{r}=O(\frac{\log r}{r}).$$
So subadditivity of $TV_r$ under connected sum follows again from the subadditivity of $\limsup.$
The claims about $lTV$ follow similarly as $\liminf$ is superadditive.
\end{proof}
\vskip 0.08in

 We are now ready to finish the proof of the main result of the paper that was stated as Theorem \ref{LTVbound} in the introduction. We slightly restate the theorem.
 
 \begin{named}{Theorem \ref{LTVbound}} 
There exists a universal constant $C$ such that for any compact orientable $3$-manifold $M$ with empty or toroidal  boundary we have
$$LTV(M)\leqslant C ||M||,$$
where the constant $C$ is about  $8.3581 \times 10^9$.
\end{named}
\begin{proof}
As both $LTV(M)$ is subadditive and $||M||$ additive under disjoint union and connected sum, it is enough to prove it for prime manifolds. As $TV_r(S^2\times S^1)=1$ and $||S^2\times S^1||=0,$ we can ignore $S^2 \times S^1$ factors.

Next we  treat the case of hyperbolic manifolds. By Theorem \ref{Thurstontriangulation}, if $M$ is hyperbolic, there is a link $L$ in $M$ such that $M\setminus L$ admits a triangulation with at most $C_2 ||M||$ tetrahedra, where $C_2$ is the universal constant defined in \ref{sec:triangulations}.
 By Corollary \ref{dehnfilling} and Proposition \ref{nbtetrahedrabound}, we have
$$LTV(M)\leqslant LTV(M\setminus L)\leqslant  C_1 C_2 ||M||.$$
Setting $C=C_1C_2$ we recall that $C_1$ has
been estimated to be less than $1.101 \times {10}^{9}$ in Subsection \ref{sec:triangulations}.
Furthermore
$C_1=v_8 +8 \Lambda \left( \frac{\pi}{8} \right) $ which is about 7.5914.
Thus the constant $C=C_1 C_2$ is about  $8.3581 \times {10}^{9}$.
\vskip 0.07in

By Theorem \ref{TVSeifert}, $LTV(M)=0$, if  $M$ is a Seifert fibered manifold. As $M$ is a Dehn-filling of $\Sigma \times S^1$ for some surface with boundary $\Sigma,$ its Gromov norm is $0$ as $||\Sigma \times S^1||=0$ by Theorem  \ref{simplicialvolproperties}. Thus the result is true in this case.

Now suppose that $M$ is any compact, oriented 3-manifold that is closed or has toroidal boundary. By the Geometrization Theorem (see Theorem \ref{GT}) there is a collection of essential, disjointly embedded tori
in ${\mathcal T}=\{T_1,\ldots, T_n\}$  in $M$, such that 
such that all the connected components of
$M\setminus {\mathcal T}$ are either Seifert fibered manifolds  or hyperbolic. 
By the above discussion the result is true for each component of $M\setminus {\mathcal T}$.
The simplicial volume is additive over the components  of $M\setminus {\mathcal T}$ (Theorem \ref{simplicialvolproperties}).
\\ By Proposition \ref{LTVconnectedsum},
$LTV$ is  subadditive over  the components  of $M\setminus{\mathcal T}$.
Applying 
Theorem \ref{toruscutting} inductively we get
$$LTV(M)\leqslant LTM(M\setminus {\mathcal T})\leqslant C||M\setminus {\mathcal T}||= C||M||,$$
where $C=C_1 C_2$ is about  $8.3581 \times {10}^{9}$.
This concludes the proof of Theorem \ref{LTVbound}.
\end{proof}
\vskip 0.05in

 Next we discuss lower bounds for the Turaev-Viro invariants. 
 \begin{corollary}\label{reduces} 
Let  $M, M'$ be compact, oriented 3-manifolds with empty or toroidal boundary  and such that $M$ is obtained by Dehn filling  from $M'$ and suppose that $lTV(M) >0$.
 Then we have
  $$lTV(M') \geqslant lTV(M)>0.$$
  \end{corollary}

 \begin{proof} Since $M$ is obtained by Dehn filling from $M'$, Corollary \ref{dehnfilling}
gives  $TV_r(M)\leqslant TV_r(M'),$ and thus $lTV(M'))\geqslant  lTV(M)$. 
\end{proof}

Corollary \ref{reduces} applies in particular when $M'$ is a knot complement in $M$; this application also gives the proof of Corollary \ref{expgrowth}.

\begin{proof} {\rm (of Corollary \ref{expgrowth})} Let $K\subset S^3$ be the  figure-8 knot or the  Borromean rings. By  \cite[Corollary 5.2]{DKY},  for $M_K=S^3\setminus K$,
we have $lTV(M_K)=v_3||M_K||=vol(M_K)\geqslant 2v_3$, where the last inequality  follows from the fact that
the volume of the figure-8 knot complement is $2v_3$ while the volume of the Borromean rings complement is $2v_8$.
\\ If $L$ is a link in $S^3$ that contains $K$, then $M_K$ is obtained by Dehn filling from $M_L=S^3\setminus L.$ 
By Corollary \ref{reduces} we have $lTV(M_L) \geqslant 2v_3.$
\end{proof}

 Removing solid tori from a 3-manifold  can also be thought of as a special case of removing a Seifert fibered sub-manifold.
 This generalized operation also preserves exponential growth of the $TV_r$.
 \begin{corollary}\label{expgrowth-3}Let $M$ be a compact oriented $3$-manifold such that $TV_r(M)$ grows exponentially, that is $lTV(M)>0.$ Assume that $S$ is a Seifert manifold embedded in $M.$ Then 
 $$lTV(M\setminus S)\geqslant lTV(M)>0.$$
 \end{corollary}
\begin{proof}
Note that $\partial S$ consists of $n\geqslant 1$ tori. By Theorem \ref{toruscutting}, 
$$TV_r(M)\leqslant \left(\frac{r-1}{2}\right)^n TV_r(M\setminus S) TV_r(S).$$ But by Theorem \ref{TVSeifert}, there are constants $A>0$ and $N$ such that $TV_r(S)\leqslant A r^N.$ Thus 
$$TV_r(M\setminus S)\geqslant A' r^{-N'} TV_r(M)$$
for some constants $A'>0$ and $N'$, and $lTV(M\setminus S)\geqslant lTV(M)>0.$
\end{proof}

\subsection{Sharper estimates and  Dehn filling} \label{twist} 
In this section, we will use results of Futer, Kalfagianni and Purcell \cite{fkp:filling}  to obtain much sharper relations between $LTV$ and volumes of hyperbolic  link complements.
We will also  address the question of the extent to which relations between Turaev-Viro invariants of hyperbolic volume survive under Dehn filling.
To state our results we need some terminology that we will not define in detail.  For definitions and  more details the reader is referred to  \cite{FKP}.

Over the years there has been a number of results about  coarse relations between  diagrammatic link invariants and the  volume of hyperbolic links. See \cite{FKP} and references therein.
Using such results, for restricted classes of 3-manifolds,  we obtain sharper bounds than the one of Theorem \ref{LTVbound}.

  A \emph{twist region} in a diagram is a portion of the diagram consisting of a maximal string of bigons arranged end-to-end, where maximal means there are no other bigons adjacent to the ends. We require twist regions to be alternating.
The number of twist regions is the \emph{twist number} of the diagram, and is denoted $\mathrm{tw}(D)$. 

For a link $L$
  in $S^3$ with a diagram $D$  with $tw(D)$ twist regions, the augmented link $L'$ of $L$ is obtained  by adding a crossing circle around each twist region and replacing the twist region by two parallel strands with at most one crossing. See, for example, \cite[figure 2]{FKP}. The complement of  $L$ can  be obtained  from the complement of  $L'$ by Dehn-filling along the boundary components corresponding to the crossing circles. 
    
  \begin{theorem}\label{twistnumberbound} Let $L$ be a link  $S^3,$ that admits a prime, twist reduced diagram{\footnote{ Every prime knot has prime twist reduced diagrams}}  $D$ with $\mathrm{tw}(D)>1$,
  and such that each twist region has at least $n\geq 7$ crossings.
Then $L$ is hyperbolic and  we have
$$LTV(S^3\setminus L)\leqslant  10.4 \left(1-\left(\frac{2\pi}{\sqrt{n^2+1}}\right)^2\right)^{-3/2}      \ \vol(S^3\setminus L) .$$

\end{theorem}
\begin{proof}
By a result of Agol and D. Thurston  \cite[Appendix]{lackenby:alt-volume})
the complement of the augmented link $L'$ obtained from $D$, has a triangulation with at most $10 (tw(D)-1)$ ideal tetrahedra. 
Thus  by Theorem \ref{nbtetrahedrabound}
$$LTV(S^3\setminus L')\leqslant  (2.08) \cdot 10   v_8 \  (tw(D)-1) =  (10.4) \cdot 2  v_8 \  (tw(D)-1).$$

The complement of the link $L$ is obtained by  Dehn-filling from $S^3\setminus L'.$ 
In fact, Futer and Purcell \cite[Theorem 3.10]{futer-purcell} show that if each twist of $D$ has at least  $n$ crossings, then all the filling slopes for the Dehn-filling from  
$S^3\setminus L',$  to $S^3\setminus L, $  have length at least $\sqrt{n^2+1}$.
Suppose that the diagram $D$ of  $L$ has at least $n$ crossings per twist region for some  $n\geq 7$.

Then
by  \cite[Theorem 1.2]{fkp:filling} and its proof, we have

$$ 2 \ v_8 (tw(D)-1))  \leqslant \  \left(1-\left(\frac{2\pi}{\sqrt{n^2+1}}\right)^2\right)^{-3/2} \  \vol(S^3\setminus L). $$
 
By Corollary \ref{dehnfilling}, we have
$$LTV(S^3\setminus L)\leqslant LTV(S^3\setminus L').$$

Now combining the three last inequalities we get the desired result.

\end{proof}
 
 \begin{remark}Theorem \ref{twistnumberbound} says that for most links we have $LTV(S^3\setminus L)\leqslant  10.5   \ \vol(S^3\setminus L)$:
 Indeed, for links which at least $n$ twists per twist region as above, for $n$ large enough the inequality is satisfied. Then  for every $B>0$, for links $L$ that admit diagrams with $tw(D)\leq B,$ for a generic choice of the number of twists in each twist region, the inequality is satisfied.
\end{remark}
  \smallskip
 
 To continue, let  $M$ be a compact 3-manifold with toroidal boundary whose interior is hyperbolic, and let  $T_1, \ldots, T_k$ be some components of $\partial M$.
On each $T_i$, choose a slope $s_i$, such that the shortest length of any of the $s_i$ is $\lmin > 2\pi$. 
  Then the manifold $M(s_1, \dots, s_k)$ obtained by Dehn filling along $s_1, \dots, s_k$ is hyperbolic and \cite[Theorem 1.1]{fkp:filling}
  gives a correlation between its volume and the volume of $M$.

The next result provides some information on how relations between Turaev-Viro invariants and hyperbolic volume behave  under Dehn 
 filling.

\begin{corollary} \label{preserve} Let $M$ be a compact 3-manifold with toroidal boundary whose interior is hyperbolic and let the notation be as above.
Suppose that $LTV(M)=\vol (M)$.  For $\lmin > 2\pi$ we have
$$LTV(M(s_1, \dots, s_k))\  \leqslant \  B({\lmin}) \  \vol(M(s_1, \dots, s_k)), $$
where $B({\lmin})$ is a function that approaches 1 as $\lmin \to \infty$.
\end{corollary}
\begin{proof} 
Since $\lmin \to \infty$, Theorem \cite[Theorem 1.1]{fkp:filling}  applies to give 
$$\left(1-\left(\frac{2\pi}{\lmin}\right)^2\right)^{3/2}   \     \vol (M) \leqslant \vol(M(s_1, \dots, s_k)).$$

Combining the last inequality with Corollary \ref{dehnfilling} we have
$$LTV(M(s_1, \dots, s_k))\  \leqslant LTV(M)=\vol (M) \leqslant \left(1-\left(\frac{2\pi}{\lmin}\right)^2\right)^{-3/2} \vol(M(s_1, \dots, s_k)).$$

\end{proof}

By \cite[Theorem 1.6]{DKY} if $M$ is the complement of the figure-8 knot or the Borromean rings $B$ we have
 $LTV(M)=\vol (M)$.  Let $K_n$ denote the double twist knot obtained by $1/n$-filling along each of two components of $B$.
 By Corollary \ref{preserve}, for any constant $E$ arbitrarily close to 1, there is $n_0$ so that $LTV(S^3\setminus K_n)\leqslant\ E \  \vol((S^3\setminus K_n)$ whenever $n\geqslant n_0.$

 %%%%%%%%%%%%%%%%

\section{Exact calculations of Gromov norm from Turaev-Viro invariants} \label{Exact} 

In this section we give two examples of families of manifolds $M$ where the growth rate of Turaev-Viro invariants gives the Gromov norm $||M||$ exactly. Both examples are derived as applications of the results in Section \ref{sec:cuttingtori}, \ref{sec:TVSeifert} and \ref{sec:LTVbound}. Both results provide partial verification of the following.

\begin{conjecture}\label{TVvolumeconj}{\rm (Turaev-Viro  invariants volume conjecture, \cite{Chen-Yang})}
For every compact orientable  3-manifold $M,$ with empty or toroidal boundary, we have
$$LTV(M)=\underset{r\to \infty}{\limsup} \frac {2\pi}  {r} \log |TV_r(M)|=v_3 ||M||,$$
where $r$ runs over all odd integers.
\end{conjecture}

A stronger version of conjecture \ref{TVvolumeconj} was first stated by Chen and Yang \cite{Chen-Yang}  for hyperbolic manifolds only and was  supported  by experimental evidence.
The version above, which is the natural generalization of the conjecture in \cite{Chen-Yang} was stated 
in \cite{DKY} for links in $S^3$, where the authors and Yang also gave the first examples where the conjecture is rigorously verified.
In particular, they proved it for knots  in $S^3$ of simplicial volume zero.
Here, as a corollary of  Theorem \ref{LTVbound} and Corollary \ref{dehnfilling} we generalize this later result as follows.

 \begin{corollary}\label{zerogen} Suppose that $M$ is a compact, orientable 3-manifold with $lTV(M)\geqslant 0$.
 Then, for any link   $K\subset M$ with $||M\setminus K||=0$, we have

$$lTV(M)=LTV(M)=\lim_{r\to \infty} \frac {2\pi} {r} {\log |TV_r(M\setminus K)|} =v_3 ||M\setminus K||=0,$$
where $r$ runs over all odd integers. That is Conjecture \ref{TVvolumeconj} holds for $M\setminus K.$

In particular, the conclusion holds if $M=S^3$ or $\displaystyle{\#(S^1\times S^2)^k.}$
 \end{corollary}
 \begin{proof}
 By   Theorem \ref{LTVbound}, $LTV(M \setminus K)\leqslant 0.$ By Corollary \ref{dehnfilling},
 $$TV_r(M)\leqslant TV_r(M\setminus K),$$
and thus $0\leqslant lTV(M)\leqslant LTV(M\setminus K)\leqslant 0.$ Thus the conclusion follows.

The claim about $S^3$ or $\displaystyle{\#(S^1\times S^2)^k}$ follows since,  as it is easily seen by Theorem \ref{TQFT}(2), we have $lTV(S^3), lTV(S^1\times S^2)\geqslant 0$,
and $lTV$ is superadditive under connected sums.
 \end{proof}
\vskip 0.08in

To describe our second family of examples,
we introduce an operation we call \emph{invertible cabling} that leaves both the Gromov norm and the growth rate of Turaev-Viro invariants unchanged.
\begin{definition}\label{def:invertible}A manifold $S$, with  $||S||=0$ and with a distinguished torus boundary component $T$, is called an invertible cabling space if
there is a Dehn-filling on some components of $\partial S\setminus T$ that is homeomorphic to $T \times
[0,1].$
\end{definition}
A way to obtain invertible cabling spaces is to start with a link $L$  in a solid torus 
such that $L$ contains at least one copy of the core of the solid torus. One example of such a cabling space $S$ is the complement in a solid torus of $p\geqslant 2$ parallel copies of the core.

Using  Corollary \ref{dehnfilling} we also show the following: 
\begin{corollary}\label{thm:invertible} Let $M$ be a 3-manifold with toroidal boundary for which Conjecture \ref{TVvolumeconj} holds and $S$ be an invertible cabling space.
Let $M'$ be obtained by gluing a component of $\partial S\setminus T$ to a component of $\partial M.$ Then, we have
$$LTV(M)=lTV(M)=lTV(M')=LTV(M'),$$
and thus Conjecture  \ref{TVvolumeconj}  holds for $M$.
\end{corollary}
\begin{proof}
As there is a Dehn-filling on components of $\partial S\setminus T$ that is homeomorphic to $T \times [0,1],$ $M$ is a Dehn-filling of $M'$ and $TV_r(M)\leqslant TV_r(M')$ by Corollary \ref{dehnfilling}. 
\\ On the other hand, $M'$ is obtained by gluing $S$ to $M$ along a torus. By Corollary \ref{dehnfilling} again, 
$$TV_r(M')\leqslant TV_r(M)TV_r(S).$$
 But as $S$ has volume $0,$ by Theorem \ref{TVSeifert}, we know that there exists constants $A$ and $N$ such that 
$$TV_r(S)\leqslant A r^N.$$
 \\ On the other hand, we also have that $||M'||=||M \underset{T}{\cup} S|| \leqslant || M|| +||S||=||M||$ by Theorem \ref{simplicialvolproperties}, and also $||M|| \leqslant ||M'||$ as $M$ is a Dehn-filling of $M'.$ Thus $M$ and $M'$ have the same simplicial volume too.
\end{proof}
\vskip 0.07in

Corollary \ref{thm:invertible} 
applies in particular for $M$ the complement of the figure-8 knot   or to links  with complement homeomorphic to the complement of the Borromean rings. Furthermore, it applies to complements of fundamental shadow links
for which Conjecture \ref{TVvolumeconj}  was verified in \cite{BDKY}.

\appendix
\section{Proof of Lemma \ref{maxfunct6j}}
We prove Lemma \ref{maxfunct6j}  which we use to get the upper bound for $6j$-symbols in Section \ref{sec:6jbound}.
\begin{lemma}\label{maxfunct6j} The maximum of the function $v$ is $\displaystyle {\frac{v_8}{4}}$ and the maximum of the function $g$ is $\displaystyle{8\Lambda \left( \frac{\pi}{8}\right)}$
\end{lemma}
\begin{proof}
 The function $v$ is differentiable and $\pi$-periodic in all variables, so such a maximum exists and is a critical point of $v$.
Computing the partial derivatives $\displaystyle {\frac{\partial v}{\partial \alpha}},$  $\displaystyle {\frac{\partial v}{\partial \beta}},$ and  $\displaystyle {\frac{\partial v}{\partial \gamma}},$ we see that $(\alpha,\beta,\gamma)$ is a critical point of $v$ if and only if
$$\begin{pmatrix}
1 & -1 & -1 & 1 \\
1 & -1 & 1 & -1 \\
1 & 1 & -1 & 1
\end{pmatrix}\begin{pmatrix}
\Lambda'(\alpha+\beta+\gamma) 
\\ \Lambda'(\alpha+\beta-\gamma)
\\ \Lambda'(\alpha+\gamma-\beta)
\\ \Lambda'(\beta+\gamma-\alpha)
\end{pmatrix}=0.$$
The matrix $\begin{pmatrix}
1 & -1 & -1 & 1 \\
1 & -1 & 1 & -1 \\
1 & 1 & -1 & 1
\end{pmatrix}$ has rank $3$ and kernel $\textrm{Vect}\begin{pmatrix}
1 \\ 1 \\ 1 \\ 1
\end{pmatrix}$. Hence $(\alpha,\beta,\gamma)$ is a critical point of $v$ if and only if
$$\Lambda'(\alpha+\beta+\gamma) 
= \Lambda'(\alpha+\beta-\gamma)
= \Lambda'(\alpha+\gamma-\beta)
= \Lambda'(\beta+\gamma-\alpha),$$
which, given that $\Lambda'(x)=-\log (|2\sin x|)$, is equivalent to
$$|\sin(\alpha+\beta+\gamma)|
=| \sin(\alpha+\beta-\gamma)|
= |\sin(\alpha+\gamma-\beta)|
= |\sin(\beta+\gamma-\alpha)|.$$
This means that the angles $\alpha+\beta+\gamma,$ $\alpha+\beta-\gamma,$ $\alpha+\gamma-\beta$ and $\beta+\gamma-\alpha$ are all equal or opposite mod $\pi.$ Let us write $$x=\alpha+\beta+\gamma=\pm (\alpha+\beta-\gamma)=\pm (\alpha+\gamma-\beta)=\pm (\beta+\gamma-\alpha) (\mathrm{mod} \ \pi).$$ 

Given that $\Lambda$ is $\pi$ periodic and odd, at such a critical point we have $v(\alpha,\beta,\gamma)=\frac{n}{2}\Lambda(x)$, where $n$ is an even integer between $-2$ and $4.$ Moreover we have $v(\alpha,\beta,\gamma)=2\Lambda(x)$ if and only if 
$$\alpha+\beta+\gamma=- (\alpha+\beta-\gamma)=- (\alpha+\gamma-\beta)=- (\beta+\gamma-\alpha) (\mathrm{mod} \ \pi).$$
This system is equivalent to $\displaystyle {\alpha=0 (\mathrm{mod} \ \frac{\pi}{4})}$ and $\displaystyle {\beta=\gamma=\alpha (\mathrm{mod} \ \frac{\pi}{2})}.$ We then see that the maximal value of $v$ at such a critical point is $\displaystyle {2\Lambda \left(\frac{\pi}{4}\right)=\frac{v_8}{4}}$. 
This value is obtained for $\displaystyle {\alpha=\beta=\gamma=\frac{3\pi}{4}}$ or if two of the angles $\alpha,\beta,\gamma$ are equal to $\frac{\pi}{4}$ and the last one is $\frac{3\pi}{4}.$

 For other critical points, where $\displaystyle {v(\alpha,\beta,\gamma)=\frac{n}{2}\Lambda(x)}$ with $|n|\leqslant 2,$ the value of $v$ is bounded by $\displaystyle {\Lambda \left(\frac{\pi}{6}\right)}$ as $\displaystyle {\Lambda\left(\frac{\pi}{6}\right)=\frac{3}{2}\Lambda\left(\frac{\pi}{3}\right)=\frac{v_3}{2}}$ is the maximum of $\Lambda.$ 
 
 But we have that 
$\displaystyle {v_3\simeq 1,01494\ldots  < \frac{v_8}{2}=1.83419\ldots}$. So the maximum of $v$ is $\displaystyle {\frac{v_8}{4}}$.

Similarly, we see that $(Z,A_1,A_2,A_3,A_4,A_5,A_6)$ is a critical point of $g$ if and only if 
$$\begin{pmatrix}
-1 & 1 & 1 & 1 & 1 & -1 & -1 & -1 \\
0 & -1 & -1 & 0 & 0 & 1 & 1 & 0\\
0 & -1 &  0 & -1 & 0 & 1 & 0 & 1 \\
0 & -1 & 0 & 0 & -1 & 0 & 1 & 1 \\
0 & 0 & 0 & -1 & -1 & 1 & 1 & 0 \\
0 & 0 & -1 & 0 & -1 & 1 & 0 & 1 \\
0 & 0 & -1 & -1 & 0 & 0 & 1 & 1 \\
\end{pmatrix} 
\begin{pmatrix}
\Lambda'(Z)\\
\Lambda'(Z-U_1)\\
\Lambda'(Z-U_2)\\
\Lambda'(Z-U_3)\\
\Lambda'(Z-U_4)\\
\Lambda'(V_1-Z)\\
\Lambda'(V_2-Z)\\
\Lambda'(V_3-Z)\\
\end{pmatrix}=0.$$
 The matrix $\begin{pmatrix}
-1 & 1 & 1 & 1 & 1 & -1 & -1 & -1 \\
0 & -1 & -1 & 0 & 0 & 1 & 1 & 0\\
0 & -1 &  0 & -1 & 0 & 1 & 0 & 1 \\
0 & -1 & 0 & 0 & -1 & 0 & 1 & 1 \\
0 & 0 & 0 & -1 & -1 & 1 & 1 & 0 \\
0 & 0 & -1 & 0 & -1 & 1 & 0 & 1 \\
0 & 0 & -1 & -1 & 0 & 0 & 1 & 1 \\
\end{pmatrix}$ has rank $7$ and kernel $\mathrm{Vect}\begin{pmatrix}
1 \\
1 \\
1 \\
1 \\
1 \\
1 \\
1 \\
1
\end{pmatrix}$. 

Hence $(Z,A_1,A_2,A_3,A_4,A_5,A_6)$ is a critical point of $g$ if and only if
\begin{multline*}\Lambda'(Z)=
\Lambda'(Z-U_1)=
\Lambda'(Z-U_2)=
\Lambda'(Z-U_3)=
\Lambda'(Z-U_4) \\=
\Lambda'(V_1-Z)=
\Lambda'(V_2-Z)=
\Lambda'(V_3-Z),\end{multline*}

which is equivalent to

 \begin{multline*}Z=\pm (Z-U_1)=\pm(Z-U_2)=\pm(Z-U_3)=\pm (Z-U_4)
\\ =\pm(V_1-Z)=\pm(V_2-Z)=\pm(V_3-Z) (\mathrm{mod} \ \pi).
\end{multline*}
 
 As above, the  function being $\pi$-periodic and odd, at such a critical point we will have
$$g(Z,A_1,A_2,A_3,A_4,A_5,A_6)=-n\Lambda(Z),$$
with $n$ an even integer between $-6$ and $8$. Furthermore, $n=8$ if and only if we have
$$Z=-(Z-U_i)=-(V_j-Z) \ (\mathrm{mod} \ \pi).$$
From this we get $U_i=2 Z \ (\mathrm{mod} \ \pi)$ and $V_j=0  \ (\mathrm{mod} \ \pi)$. But, as 
$$U_1+U_2+U_3+U_4=V_1+V_2+V_3,$$ we have that $8 Z=0  \ (\mathrm{mod} \ \pi).$
 
 Finally, as $\displaystyle {\Lambda \left(\frac{\pi}{8} \right)\simeq 0.490936>0.457982\simeq\Lambda \left(\frac{\pi}{4}\right)}$ and 
 $$8\Lambda \left(\frac{\pi}{8}\right) \simeq 3.927488 >3v_3=6\Lambda\left( \frac{\pi}{6}\right) \simeq 3,0448,$$ the maximum value of $g$ is $\displaystyle {8\Lambda\left( \frac{\pi}{8}\right)}.$

Notice that this maximum is attained for $\displaystyle {Z=\frac{7\pi}{8}}$ and either all $\displaystyle {A_i=\frac{\pi}{4}},$ or all $A_i$ are equal to $\displaystyle {\frac{3\pi}{4}}$ mod $\pi,$ except two corresponding to opposite edges in the tetrahedron which are equal to $\displaystyle {\frac{\pi}{4}}.$
\end{proof}
\bibliographystyle{hamsplain}
\bibliography{biblio}

\end{document}